\documentclass[11pt]{article}

\usepackage{amsmath,amsthm,amsfonts,amssymb}
\usepackage{enumerate,graphicx,psfrag,subfigure}
\usepackage[usenames]{color}
\usepackage{hyperref}
\usepackage{soul}

\newtheorem{theorem}{Theorem}[section]

\newtheorem{lemma}[theorem]{Lemma}
\theoremstyle{definition}

\newtheorem{corollary}[theorem]{Corollary}

\newtheorem{example}[theorem]{Example}

\newtheorem{remark}[theorem]{Remark}

\def\b{{\beta}}
\def\A{{\mathcal{A}}}
\def\a{{\alpha}}
\def\d{{\delta}}
\def\g{{\gamma}}
\newcommand{\fracd}[2]{{\displaystyle\frac{#1}{#2}}}

\title{Double cosets in free groups}
\author{Elizaveta Frenkel \\ Vladimir N. Remeslennikov}

\begin{document}
\maketitle

%
%

%
%





\begin{abstract}
In this paper we study double cosets of finite rank free groups.
We focus our attention on cancellation types in double cosets and
their formal language properties.
\end{abstract}

Mathematics Subject Classification 2010: 20E05, 20F10, 20F65

Key words: double cosets, subgroups of free groups, regular
subsets, bounded cancellations.



\section{Introduction}

By the classical Nielsen-Schreier theorem every subgroup of a free
group is again a free group. One of the goals of this paper is to
study how one finitely generated free group $C$ can be embedded
into another free group $F$. Clearly, the type of embedding of a
subgroup $C$ into a group $F$ affects  left and right cosets, and
double cosets of type $CfC$, $f \in F$. A description of
combinatorial and formal language properties of these complexes is
the second goal of the present paper.

Similar questions have been studied widely as auxiliary ones, for
example, in papers \cite{2}, \cite{12}, \cite{15} in relation
to hyperbolicity of amalgamated products of groups, and in
\cite{3}, \cite{7}, \cite{8} in connection with the
conjugacy problem in amalgamated products.

There exist a lot of papers about the characterization of
embeddings of a subgroup $C$ in a group $F$, but we shall mention
only the closest to the subject of our paper. In \cite{1} was
shown that the property of a finitely generated subgroup $C$ to be
malnormal is algorithmically decidable in a free group $F$,
meanwhile \cite{4} provides an example of the undecidability of
this problem in the class of hyperbolic groups. Another related
result was proved in \cite{10}, where it was shown that a
random finitely generated subgroup of a free group $F$ is
malnormal.

The essential role in these studies is played by the structure of a Cayley graph of a group and
Schreier graphs for a subgroup $C$ in $F$. In particular, the embedding of $C$ in $F$ can be described in terms of  equations $E(f,f)$ and $E(f,g)$ (see section \ref{sub_basic}), lemmas \ref{le:e(gg)}, \ref{le:e(fg)}). In this paper we obtain
results on double cosets in a free group $F$ with respect to the type of embedding of a subgroup $C$ in $F$.
The main results are the following theorems:

{\bf Theorem \ref{th:bounded}.} { \it Let $F$ be a free group with a
finite basis $X$, and $C$ be a finitely generated subgroup of $F$
with a Nielsen set of generators $Y$. Suppose $f \in F
\smallsetminus C$ and $f$ is of minimal length in a double coset
$CfC$. Then there exists a natural number $k$ which can be
effectively determined by $C$ and $Y$ such that

\begin{itemize}
\item if $C_f = C \bigcap f^{-1} C f = 1$, then
multiplication in the complex $CfC$ is $k-$reduced, which means
$l_X(c_1) + l_X(f) + l_X(c_2) - l_X(c_1 f c_2) \le 2 k$, where
$c_1, c_2 \in C$ and $l_X(\cdot)$ is the length of elements in
$F(X)$.

\item if $C_f \neq 1$, then $CfT$ is $k-$reduced, where $T$ is
a relative Schreier transversal of $C$ in $F$ relative to $C_f$.
\end{itemize}
}

{\bf Theorem \ref{th:cfcreg}.} { \it Let $C$ be a finite rank subgroup of a free
group $F(X)$. Then the set of all $X-$reduced words representing
elements of $CfC$ is regular in $F(X)$. Moreover, an automaton which accepts this set can be constructed effectively by $C$ and $f$ if $C$ is $f-$malnormal, or by $C,f$ and a relative Schreier transversal $T$ of $C$ with respect to $C_f$, otherwise.}

\section{Preliminaries}\label{Section:preliminaries}

In this section, following \cite{6,7,11}, we give some basics
on regular subsets, cosets, and Schreier representatives of
finitely generated subgroups of a finite rank free group $F$. We
shall use the following notation throughout the paper.

Let $F = F(X)$ be a free group with basis $X=\{x_1,\dots, x_m\}$.
We identify  elements of $F$ with  reduced words  in the alphabet
$X \cup X^{-1}$ (i.e. \emph{$X-$reduced } words), and let $l_X(f)$
denote the length of an element $f \in F$, i.e. the number of
letters of the reduced word $f$ in $X \cup X^{-1}$.

Fix a subgroup $C = \langle h_1, \ldots, h_r \rangle$ of $F$
generated by finitely many elements $h_1, \ldots, h_r  \in F$ and
an arbitrary element $f \in F$.

The set $CfC = \{c_1 f c_2 | c_1, c_2 \in C \}$ is called a
\emph{double coset} and $f$ is a \emph{representative} of this
coset.

\subsection{Regular subsets in a free group and finite state automata}\label{subs_reg}

Suppose we are given a finite set $X$, which we shall call the {\em state set}. An {\em arrow} is a triple $(s_1,x,s_2)$, where $s_1, s_2$
are elements of $S(\A)$ and $x$ is an element of $X \cup
\varepsilon$ and is called the {\em label} of the arrow. The {\em
source} of the arrow is $s_1$ and its {\em target} is $s_2$ . An arrow labelled $x$ is sometimes called an $x-${\em transition}.

A {\em non-deterministic finite state automaton} $\A$ is a
quintuple $(S(\A),X,\d,S_0,F_0)$, where $S(\A)$ is a finite set of
{\em states}, $X$ is a finite set called the {\em alphabet}, $S_0 \subset S(\A)$ is the
(non-empty) set of {\em initial states}, $F_0 \subseteq S(\A)$
is the set of {\em final states}, and $\d$ is a set of arrows with labels in the enlarged alphabet $X \cup \varepsilon$. Here $\varepsilon$ is assumed
not to lie in $X$.

By a path of arrows in $\A$ we mean a sequence $(s_1, u_1, s_2,
\ldots, u_n, s_{n+1})$, where $n \ge 0$ and each $u_i$ is an arrow
with source $s_i$ and target $s_{i+1}$, for $1 \le i \le n$. We
call $s_1$ the source and $s_{n+1}$ the target of the path of
arrows. Let $w_i$ be the label of $u_i$, i.e. a letter of $X$ or
$\varepsilon$. Let $w$ be the concatenation $w_1 \cdots w_n$; if
$n = 0$, set $w = \varepsilon$. Then $w$ is called the {\em label}
of the path. Let $X^{\ast}$ be a monoid of strings over $X$ with
identity $\varepsilon$. The {\em language } $L(\A)$ over $X$
assigned to a non-deterministic automaton is the set of elements
of $X^{\ast}$ obtainable from the labels of all possible paths of
arrows with source in $S_0$ and target in $F_0$; in this case we
say that $L= L(\A)$ is a language {\em accepted} by $\A$. We say
$w$ is {\em readable } in $\A$ from $s_1$ to $s_2$, where $s_1,
s_2 \in S(\A)$, if there is a path of arrows from $s_1$ to $s_2$
with the label $w$.

A {\em deterministic finite state automaton} can be considered a
special case of a non-deterministic finite state automaton, for
which the following conditions are satisfied: there are no
arrows labelled $\varepsilon$; each state is the source of exactly
one arrow with any given label from $X$; and $S_0$ has exactly one
element.

By the Kleene-Rabin-Scott theorem a regular language over $X$ may be
identified with the language accepted by some non-deterministic
finite state automaton, or, equivalently, by some deterministic
finite state automaton. If $R = L(\A)$ is a subset of a group $F$,
then $X$ is supposed to be a set of group generators for $F$, i.e.
elements of $R$ are meant to be reduced strings over $X \cup
X^{-1}$.

\subsection{Graphs associated with a subgroup and Schreier transversals}\label{subs_graphs}

We shall sometimes identify an automaton $\A$ with a finite
connected oriented labelled graph $\Gamma$ with a distinguished
vertex in natural way. Namely, take vertex set $V=V(\Gamma)=S(\A)$; edge set $E=E(\Gamma)$ corresponding to the arrows, with
induced labelling, and form subsets $S_0, F_0 \subseteq S(\A)$
from distinguished vertices. We shall also ascribe properties of
automata to graphs.

One can associate with a subgroup $C$ two graphs: \emph{the
subgroup graph} $\Gamma =\Gamma_C$ and the Schreier graph
$\Gamma^\ast = \Gamma^\ast_C$; we refer reader to \cite{11} for
more details on this subject.

Recall, that $\Gamma$  is a finite connected digraph with edges
labelled by elements from $X$ and a distinguished vertex
(based-point) $1$ (so $S_0 = F_0 = \{ 1 \}$),  satisfying the
following two conditions. Firstly,  $\Gamma$ is folded, i.e.,
there are no two edges in $\Gamma$ with the same label and having
the same initial or terminal vertices. Secondly, $\Gamma$ accepts
precisely the reduced words in $X \cup X^{-1}$  that belong to
$C$. To explain the latter observe, that walking along a path $p$
in $\Gamma$ one can read a word $\ell(p)$ in the alphabet $X \cup
X^{-1}$, the label of $p$, (reading $x$ in passing an edge $e$
with label $x$ along the orientation of $e$, and reading $x^{-1}$
in the opposite direction). We say that $\Gamma$ accepts a word
$w$ if $w = \ell(p)$ for some closed path $p$ that starts at $1$
and has no backtracking. Clearly, $\Gamma$ can be identified with
a deterministic finite state  automata with $1$ as the unique
initial and final state.

The {\em Schreier graph} $\Gamma^\ast = \Gamma_C^\ast$ of $C$  is
a connected labelled digraph with the set $\{Cu \mid u \in F\}$ of
the right cosets of $C$ in $F$ as the vertex set, and such that
there is an edge from
 $Cu$ to $Cv$ with a label $x \in X$ if and only if $Cux = Cv$.  One can describe the Schreier graph $\Gamma^\ast$ as  obtained from $\Gamma$
by the following procedure. Let $v \in \Gamma$ and $x \in X$ such
that there is no outgoing or incoming
 edge at $v$ labelled by $x$. For every such $v$ and $x \in X$ we attach
to $v$ a new edge $e$ (correspondingly, either outgoing or
incoming)  labelled $x$ with a new terminal vertex $u$ (not in
$\Gamma$). Then we
 attach to $u$ the Cayley graph $C(F,X)$ of $F$ relative to $X$ (identifying $u$ with the vertex $1$ of $C(F,X)$), and
 then we fold the edge $e$ with the corresponding edge in $C(F,X)$ (that is labelled $x$ and is incoming to $u$).
Observe, that for every vertex $v \in \Gamma^\ast$ and every
reduced word $w$ in $X \cup X^{-1}$ there is a unique path
$\Gamma^\ast$ that starts at $v$ and has the label $w$.  By $p_w$
we denote such a path that starts at $1$, and by $v_w$ the end
vertex of $p_w$. The resulting graph $\Gamma^\ast_C $ is the
Schreier graph of $C$ in $F$.

Notice that $\Gamma = \Gamma^\ast$ if and only if the subgroup $C$
has finite index in $F$. A spanning  subtree $T$ of $\Gamma$  with
the root at the vertex $1$ is called {\em geodesic}  if for every
vertex $v \in V(\Gamma)$ the unique path in $T$ from $1$ to $v$
is a geodesic path in $\Gamma$. For a given graph $\Gamma$ one can
effectively construct a geodesic spanning subtree $T$.

From now on we fix an arbitrary spanning  subtree $T$ of $\Gamma$.
It is easy to see that the tree $T$ uniquely extends to a spanning
subtree $T^\ast$ of $\Gamma^\ast$.

Let $V(\Gamma^\ast)$ be the set of vertices of $\Gamma^\ast$.
Since, in general, $\Gamma^\ast$ is infinite, we need to extend
terminology from subsection \ref{subs_reg}. For a subset $Y
\subseteq V(\Gamma^\ast)$ and a subgraph $\Delta$ of
$\Gamma^\ast$, we define the \emph{language accepted by} $\Delta$
and $Y$ as the set $L(\Delta, Y,1)$ of the labels $\ell(p)$ of
paths $p$ in $\Delta$ that start at $1$ and end at one of the
vertices in $Y$, and have no backtracking. Notice that the  words
$\ell(p)$ are reduced since the graph $\Gamma^\ast$ is folded.
Notice, that $F = L(\Gamma^{\ast}, V(\Gamma^{\ast}), 1)$ and  $C =
L(\Gamma,\{1\},1) = L(\Gamma^{\ast}, \{1\}, 1)$.

Sometimes we shall refer to a set of right representatives of $C$
as the {\em transversal} of $C$. Recall, that a  transversal $S$
of $C$ is termed {\em Schreier} if  every initial segment of a
representative from $S$  belongs to $S$.

Further, let $S$ be a transversal of $C$. A representative $s\in S$ is
called \emph{internal} if the path $p_s$ ends in $\Gamma$, i.e., $v_s \in V(\Gamma)$.
 By $S_{\rm int}$ we denote the
set of all internal representatives in $S.$ It follows from
definition, that $|S_{\rm int}| = |V(\Gamma)|.$

A representative $s \in S$ is called {\em geodesic}  if it
has minimal possible length in its coset $Cs$. The transversal $S$
is {\em geodesic }  if every $s \in S$ is geodesic. Clearly, if
$T^{\ast}$ is a geodesic subtree of $\Gamma^{\ast}$ (and hence $T$
is a geodesic subtree of $\Gamma$), then $S$ is a geodesic
transversal.

\section{Main results on Double Cosets}\label{section_properties}

In this section we study double cosets from different points of
view.

\subsection{Basic definitions and properties}\label{sub_basic}
Let $F=F(X)$, $C$ as above. Fix an element $f \in F \smallsetminus
C$ and consider the double coset $CfC$ for this element. Clearly,
the following conditions are equivalent:

\begin{align*}
g \in CfC &\Leftrightarrow \exists \,\, d_1, d_2 \in C\,\,\,\,\, g = d_1 f d_2\\
&\Leftrightarrow  \exists \,\, d_1, d_2 \in C\,\,\,\,\, d_1g =
fd_2.
\end{align*}

Denote by $E(g,f)$ the equation $xg = fy$ over $C$. Due to the
latter equivalence, an element $g$ belongs to $CfC$ iff the
equation $E(g,f)$ is solvable over $C$, i.e. exist $c_1, c_2 \in
C$ such that $x=c_1, y=c_2$. Lemma \ref{le:e(gg)} describes some
properties of solutions of the uniform equation $E(f,f)$.

We assign to $f$ a subgroup $C_f \subseteq F$ such that $C_f = C
\cap C^f$, where $C^f$ is the set of all elements of the form
$f^{-1}Cf$. It turns out that $C_f$ plays a special role in the
whole theory of double cosets. This subgroup is one of the most
important characteristics of the embedding of $C$ in $F$.

Further, a subgroup $C$ is called {\emph{ $f-$malnormal}} if $C_f
=1$. Recall that a subgroup $C$ is said to be {\emph{malnormal}}
if it is $f-$malnormal for all $f \in F \smallsetminus C$. If, on the
contrary, $C_f$ is nontrivial then, following \cite{1}, we call
$f$ \emph{ potentially normalizing} (or {\em pn } for short).

\begin{lemma}\label{le:e(gg)} Let $C < F$ and $f \in F \smallsetminus C$. Let $D(f,f) \subseteq C \times C$ be the set of all
solutions of equation $E(f,f)$. In the notation above
\begin{itemize}
\item $D(f,f)$ consists of pairs $(c, c^{f^{-1}})$ for all $c
\in C_f$;\\
\item if $C$ is $f-$malnormal (malnormal {\it{ a fortiori}}), then $\sharp D(f,f) = 1$;\\
\item if $g \in CfC$, then $C_g$ is isomorphic to $C_f$.
\end{itemize}
\end{lemma}
\begin{proof} To show (1) take an arbitrary $c_1 \in C_f$; then by
definition $c_1 \in C^f \cap C$. Therefore, exists an element
$c_2$ such that $c_1 = f^{-1}c_2 f$. Evidently, $(c_1,
c_1^{f^{-1}})$ provide a solution of $E(f,f)$ and all such
elements can be obtained analogously.

Let us show (2). Suppose there exist two different solutions of
$E(f,f)$, say $c_1, c_2$ and $d_1, d_2$ in $C \smallsetminus \{ 1
\}$. Then $c_1 = c_2^{f^{-1}}$, $d_1 = d_2^{f^{-1}}$ and therefore
$(c_1 d_1^{-1})^f=c_2d_2^{-1}$. From this it follows that
$c_1d_1^{-1}, c_2d_2^{-1} \in C_f = 1$. Thus $c_1 = d_1$ and $c_2
= d_2$ and the set $D(f,f)$ consists of only one pair.

To prove claim (3), suppose $g = d_1 f d_2$ for some $d_1, d_2 \in
C$. By definition $C_g = C^g \cap C = C^{d_1fd_2} \cap C =
(C^{d_1})^{fd_2} \cap C \simeq C^f \cap C^{d_2^{-1}} = C^f
\cap C = C_f$. This concludes the proof.\end{proof}

\begin{remark} If an element $f$ is pn, then every other representative in $CfC$ is also potentially normalizing by Lemma \ref{le:e(gg)}.
\end{remark}
A double coset $CfC$ is called \emph{ essential} if $C_f \neq 1$
for some representative $f \in F$ (and so for all representatives
of this coset).

\begin{lemma}\label{le:essential} The number of different essential cosets is finite.
\end{lemma}
\begin{proof} The union of all essential double cosets forms so called generalized normalizer of $C$ in $F$, i.e. the set
$N^*_F(C) = \{ f \in F | f^{-1} C f \cap C \neq 1 \}$. This union
is known to be finite (see, for example, \cite{7}, proposition
3.5.). \end{proof}
\begin{lemma}\label{le:e(fg)} Let $C < F$ and $f, g \in F \smallsetminus C$.
If the equation $E(g,f)$ has a solution $(c_1, c_2)$ over $C$,
then the set $D(g,f)$ of all solutions of this equation consist of
all pairs $(c_f c_1, c_2 c_f^{fc_2})$, where $c_f \in C_{f^{-1}}$. In
particular, if $C$ is $f-$malnormal, then there exists a unique
solution $(c_1, c_2)$ with $c_1, c_2 \in C$ such that $g = c_1 f
c_2$ for $g \in CfC$.
\end{lemma}
\begin{proof} Notice that $E(g,f)$ is consistent by assumption
 and therefore exists a pair $(c_1, c_2)$ such that $c_1g = fc_2$. If $C$ is
 $f-$malnormal,
then $C_f \simeq C_g = 1$ and $D(g,f)$ satisfies the desired
conclusion. Denote by $M$ the set of all such pairs and suppose
there exists another pair of solutions, say $(d_1, d_2)$. Then
$g^{-1} d_1^{-1}c_1 g = d_2^{-1}c_2$ and $f d_2 c_2^{-1} f^{-1} =
d_1 c_1^{-1}$, and therefore $d_2^{-1}c_2 = c_g \in C_g$ and
$d_1c_1^{-1} = c_f \in C_{f^{-1}}$.  This implies $d_1 = c_fc_1$ and $d_2
= c_2 c_g^{-1}$, but $c_g^{-1} = c_f^{c_1g} = c_f^{fc_2}$ and $(d_1, d_2) =
(c_fc_1, c_2c_f^{fc_2}) \in M$. The opposite inclusion $M
\subseteq D(g,f)$ is straightforward.\end{proof}

\subsection{Advantages of Nielsen set of generators}

Now suppose that $C$ is supplied by a Nielsen generating set. In this
subsection we analyze how this affects the cancellation in
cosets and double cosets of $C$. We also introduce here the notion
of a relative Schreier transversal of $C$ with respect to its
subgroup $D$.

Since an arbitrary set of generators for $C$ can be carried into a
Nielsen set (see, for example, \cite{11} or \cite{14}), without
loss of generality it can be assumed that $C$ is presented by such
a set $Y = \{ h_1, \ldots, h_r  \}$. Let $S$ be the geodesic
Schreier transversal for $C$ in $F$ such that the Nielsen set of
generators $Y$ has the following properties with respect to $S$.

{\bf Properties of a Nielsen set of generators:}
\begin{itemize}
\item [(i)] each $h \in Y \cup Y^{-1}$ can be written in the form
$$h = s_1 \mu(h) s_2^{-1},$$
where $s_1, s_2 \in S_{\rm int}$ are written as reduced $X-$words,
$\mu(h)$ is an element of $X \cup X^{-1}$ and
$$|l_X(s_1) - l_X(s_2)| \le 1;$$

\item [(ii)] if $h_{i_1}, h_{i_2}$ are distinct elements of $Y$,
then the letters $\mu(h_{i_1})^{\pm 1}$ and $\mu(h_{i_2})^{\pm 1}$
do not cancel on computing the $X-$reduced form of $h_{i_1}^{\pm
1}h_{i_2}^{\pm 1}$;

\item [(iii)] if $h_{i_1}, h_{i_2}, h_{i_3}$ are elements of $Y
\cup Y^{-1}$ such that $h_{i_2}\neq h_{i_1}^{-1}$ and $h_{i_2}
\neq h_{i_3}^{-1}$, then the letter $\mu(h_{i_2})$ does not cancel
on computing the $X-$reduced form of $h_{i_1}h_{i_2} h_{i_3}$.
\end{itemize}

Clearly, such $S$ always exists; and the letter $\mu(h)$ is called
the central letter of $h \in Y \cup Y^{-1}$. Set $M= M(Y) =
\left[\frac{1}{2}max\{l_X(h_1), \ldots, l_X(h_r)\}\right]+1.$

The following material (up to Lemma \ref{le:5}) can be extracted
from \cite{1}.

\begin{lemma}\label{le:35}\cite[Lemma 2]{1} Suppose $f \in F$ is pn and has minimal length in $CfC$.
If $f c_1 f^{-1} = c_2$, where $c_1, c_2 \in C, c_1\neq 1$, then at least one
of the letters of $c_1$ appears in the reduced form of $f c_1 f^{-1}$, i.e.,
in $c_2$.
\end{lemma}

\begin{lemma}\label{le:34}\cite[Lemma 3,4]{1} Suppose $f$ is of minimal length in $fC$
and that $c = h_{i_1} \ldots h_{i_n}$ (where $h_{i_1}, \ldots,
h_{i_n} \in Y \cup Y^{-1}$) is $Y-$reduced.
\begin{itemize}
\item[] Then either $\mu(h_{i_1})$ remains in $fh_{i_1} \ldots
h_{i_n}$ and in this case the cancellation with $f$ in $fc$ is
exactly that of $f$ in $fh_{i_1}$; or the central letter
$\mu(h_{i_1})$ cancels in the product $fc$, and in this case
$h_{i_1}$ is of even length and exactly half of $h_{i_1}$ cancels
with $f$;
\item[] Suppose that $\mu(h_{i_j})$ cancels in the product
$fh_{i_1} \ldots h_{i_j} \ldots h_{i_n}$, but $\mu(h_{i_{j+1}})$
does not. Then
\begin{itemize}
\item[] $j \leq r$;

\item[] for any $k = 1, \ldots, j$ the length $l_X(h_{i_k})$
is even and $\mu(h_{i_k})$ cancels in product $fc$;

\item[] for any $k = 1, \ldots, j-1$ exactly the right half
$s_{2k}^{-1}$ of $h_{i_k}$ cancels completely in $h_{i_k}
h_{i_{k+1}}$;

\item[] $l_X(h_{i_1}) < \ldots < l_X(h_{i_j})$;

\item[] if the right half of $h_{i_j}$ does not cancel
completely with
 $h_{i_{j+1}}$, then $h_{i_1} \ldots h_{i_{j-1}}s_{1j}\mu(h_{i_j})$
is precisely the part of $c$ that cancels with $f$ and
$l_X(h_{i_1} \ldots h_{i_{j-1}}s_{1j}\mu(h_{i_j})) =
\fracd{1}{2}l_X(h_{i_j})$;

\item[] if the right half of $h_{i_j}$ does cancel with
 $h_{i_{j+1}}$, then $h_{i_1} \ldots h_{i_{j}}s$
is precisely the part of $c$ that cancels with $f$ for some $s \in
S_{\rm int}$ and $l_X(h_{i_1} \ldots h_{i_j}s) \le
\fracd{1}{2}l_X(h_{i_{j+1}})$.
\end{itemize}
\end{itemize}
\end{lemma}

\begin{lemma}\label{le:5}\cite[Lemma 5]{1}
Suppose that the equation $E(f,f)$ is satisfied for $c_1, c_2 \in C$, and $f$
is of minimal length in $CfC$. Then there exists
$f' \in CfC$ of the same length as $f$ which is a product of two pieces of
generators from $Y \cup Y^{-1}$. \end{lemma}
Lemma \ref{le:5} allows us to obtain an algorithm which lists all essential double cosets:

\begin{corollary}\label{cor:listessential} There exists an algorithm $A$,
which given a subgroup $C$ of $F$ and a Nielsen set of generators
$Y = \{ h_1, \ldots, h_r  \}$ lists all essential double cosets
$CfC$ of $F$ (finite in their number by Lemma \ref{le:essential}). Moreover, $A$ runs in
polynomial time in $M = M(Y) = \left[\frac{1}{2}max\{l_X(h_1), \ldots, l_X(h_r)\}\right]+1$.
\end{corollary}
\begin{proof} The number of different pieces of generators in $Y \cup Y^{-1}$ can be estimated as $O(M^2)$, and therefore one can form a list of all possible products $f' \in CfC$ (as in \ref{le:5}) in time at most $O(M^4)$.
Let us fix an element $f \in F\smallsetminus C$ of this list. By {\cite[Theorem 1.6]{16}}, we can construct subgroup graphs $\Gamma_{C_f}$ and $\Gamma_{C}$ in time $\Theta(N+4M)$ and $\Theta(N)$, respectively, where $N \le 2M r$ and $\Theta(x) = O(x log^{\ast}(x))$. Here the function $log^{\ast} :
\mathbb{N} \rightarrow \mathbb{N}$ is given by $\log^{\ast}(x)=k$, where $k=\min\{m\in \mathbb{N}: log^m(x)\le 1\}$. The product $\Gamma_{C_f} \times \Gamma_C$ can be constructed in time $O((N+4M)N)$, which is bounded above by $O(M^2)$. If now the component of $\Gamma_{C_f} \times \Gamma_C$ containing $1_{C_f} \times 1_C$ is not an isolated vertex, then $f$ is essential.
\end{proof}

Using properties of sets of Nielsen generators and Lemmas \ref{le:35}, \ref{le:34}, we want to introduce a notion of the
relative Schreier transversal.  We shall use this to simplify
multiplication and representation of elements in double cosets.

For every pair $h_i, h_j \in Y^{\pm 1}$, $h_i \neq h_j^{-1}$,
define $a_{ij}$ to be the initial subword of $h_i$ that does not cancel on forming the product $h_ih_j$, and let $b_{ij}$ be the terminal subword of $h_j$ that does not cancel on forming $h_ih_j$.

Denote also $m_{ijk}$ the subword of $h_j$ that does not cancel on forming the triple product
$h_ih_jh_k$, where $h_j \neq h_k^{-1}, h_j \neq h_i^{-1}$. Notice
that by definition of a Nielsen set of generators all words
$a_{ij}$, $b_{ij}$ and $m_{ijk}$ in $F(X)$ are nontrivial. Denote
by $\Sigma$ this
 new alphabet $\{a_{ij}\} \cup \{b_{ij}\} \cup \{m_{ijk}\} \cup \{h_i\}$  obtained from all such generators, their pairs $h_i, h_j$ and triples $h_i, h_j, h_k$.
We also use an additional subdivision of $m_{ijk}$. Namely, let
$m_{ijk} = \a_{ij} \circ \mu_j \circ \b_{jk}$ for some $\a_{ij},
\b_{jk} \in F(X)$ (not necessarily non-trivial). (In Section
\ref{subs_weakred} we subdivide $\a$'-s and $\b$'s into smaller
pieces: $m_{ijk} = \a_{1ij} \circ \a_{2ij} \circ \mu_j \circ
\b_{2jk} \circ \b_{1jk}$).

A nontrivial $\Sigma-$reduced word $u$ is called {\emph{
$C$-admissible}} if it has one of the following forms
\begin{equation}\label{hihj1} c = a_{i_1 i_2} \circ m_{i_1 i_2 i_3} \circ \ldots \circ m_{i_{k-2} i_{k-1} i_k} \circ b_{i_{k-1}i_k} {\textrm{ or }}\end{equation}
\begin{equation}\label{hihj2}c = a_{i_1 i_2} \circ b_{i_1 i_2}  {\textrm{ or }}\\ c = h_i.
\end{equation}

Clearly, there is one-to-one correspondence between $C$-admissible
words and nontrivial $X-$reduced words $c \in C$.

Further, let $D$ be a nontrivial subgroup of $C$ and suppose $D$
is given by a finite Nielsen set of generators $Z = \{ d_1 ,
\ldots, d_m \}$, where $d_1, \ldots, d_m$ are $Y-$reduced words.
Consider a subgroup graph $\Gamma_D$, and take a maximal subtree
$\Upsilon$ of $\Gamma_D$ such that the corresponding Schreier
transversal $T$ respect the choice of central letters $\mu(d_i)$.
Denote by $T_{{\rm int}}$ the set of all inner representatives
from $T$. Combining properties (i), (ii), (iii) of a Nielsen set
of generators with formulae (\ref{hihj1}) - (\ref{hihj2}) above,
we get
\begin{align*}\label{relativedi} d_i &= t_{1} \cdot \mu(d_i) \cdot t_{2}^{-1}\\
&= (a_{i_1 i_2} \circ m_{i_1 i_2 i_3} \circ \ldots ) \circ m_{i_{j-1} i_{j} i_{j+1}} \circ ( \ldots \circ m_{i_{k} i_{k+1} i_{k+2}} \circ b_{i_{k+1}i_{k+2}})
{\textrm{ or }}\\
d_i &= a_{i_1 i_2} \circ m_{i_1 i_2 i_3} \circ b_{i_2 i_3}  {\textrm{ or }}\\
d_i &= a_{i_1 i_2} \circ b_{i_1 i_2} {\textrm{  or    }} d_i =
h_{i_1},
\end{align*}

where $t_{1}, t_{2} \in T_{{\rm int}}; \,\, \mu(d_i), h_{i_1} \in
Y \cup Y^{-1}$, and $a_{ij}, b_{ij}, m_{ijk} \in \Sigma$.

The system $T_D$ of all $C-$admissible words $t \in T$ is called
the \emph{relative Schreier transversal} of $C$ in $F$ with
respect to $D$ and $\Upsilon$.

\begin{example} Consider an example of a relative Schreier
transversal for a subgroup $C < F(a,b)$. Let $Y = \{ a^3 , b^3 , a
b^2 a^{-1} , b a^3 b^{-1} , b a b^2 a^{-1} b^{-1} \} = \{ h_1 ,
h_2 , h_3 , h_4 , h_5 \}$ be a Nielsen set of generators of $C$.
The subgroup graph $\Gamma_C$ is shown on figure \ref{fig_subgrC},
the edges entering maximal subtree are highlighted.

\begin{figure}[h!]
\begin{center}
\psfrag{h1}{$h_1$}

\psfrag{h2}{$h_2$}

\psfrag{h3}{$h_3$}

\psfrag{h4}{$h_4$}

\psfrag{h5}{$h_5$}

\psfrag{d1}{$d_1$}

\psfrag{d2}{$d_2$}

\psfrag{d3}{$d_3$}

%

\includegraphics[scale=.5]{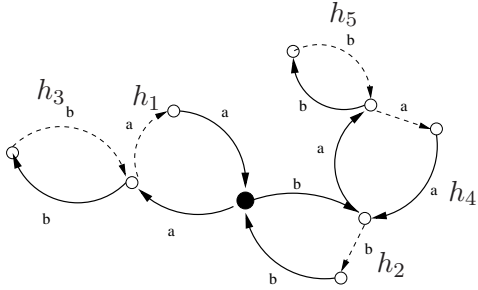}

\end{center}
\caption{Stallings automaton $\Gamma_C$ recognizing $C = \langle
a^3 , b^3 , a b^2 a^{-1} , b a^3 b^{-1} , b a b^2 a^{-1} b^{-1}
\rangle$}\label{fig_subgrC}
\end{figure}

Further, represent elements of $Y \cup Y^{-1}$ in the form $h_i =
s_{1i} \circ \mu_i \circ s_{2i}^{-1}$ (uniquely determined by
$S$):

$$h_1 = a \circ a \circ a = h_6^{-1} , \,\,\,\, \,\,\,\,\, h_2 = b \circ b \circ b = h_7^{-1},$$
$$h_3 = a b  \circ b \circ a^{-1} = h_8^{-1}, \,\,\,\,\,\,\,\,\,\, h_4 = ba  \circ a \circ a b^{-1} = h_9^{-1},$$
$$h_5 = b a b  \circ  b \circ a^{-1} b^{-1} = h_{10}^{-1}.$$

Taking all suitable products of $h_i, h_j$ and $h_k$ for $i, j, k
=1, \ldots, 10$, one can easily construct $\Sigma$. For instance,
$a_{11} = a^3$, $a_{74} = b^{-2}$, $m_{123} = b^3$, $m_{742} =
a^3$, $b_{42} = b^2$ etc.
The element $a$ is pn and a subgroup
$C_a$ is generated by a Nielsen generating set $Z = \{ h_2^{-1}
h_4 h_2, h_1 ,h_2^2 \} = \{ d_1 , d_2 , d_3\}$. Then, taking a
maximal subtree $\Upsilon$ as shown on figure \ref{fig_subgrd},
obtain $$d_1 = a_{74} \circ m_{742} \circ b_{42}, \,\,\,\,\,\,
d_2 = h_1 \,\,\,\,\,\, {\textrm{ and }} \,\,\,\,\,\,d_3 = a_{22}
\circ b_{22}.$$

\begin{figure}[h!]
\begin{center}
\psfrag{h1}{$h_1$}

\psfrag{h2}{$h_2$}

\psfrag{h3}{$h_3$}

\psfrag{h4}{$h_4$}

\psfrag{h5}{$h_5$}

\psfrag{d1}{$d_1$}

\psfrag{d2}{$d_2$}

\psfrag{d3}{$d_3$}

\psfrag{a3}{$a^3$}

\psfrag{b2}{$b^2$}

\psfrag{b4}{$b^4$}

\subfigure[Subgroup graph $\Gamma_D$]{
\includegraphics[scale=.5]{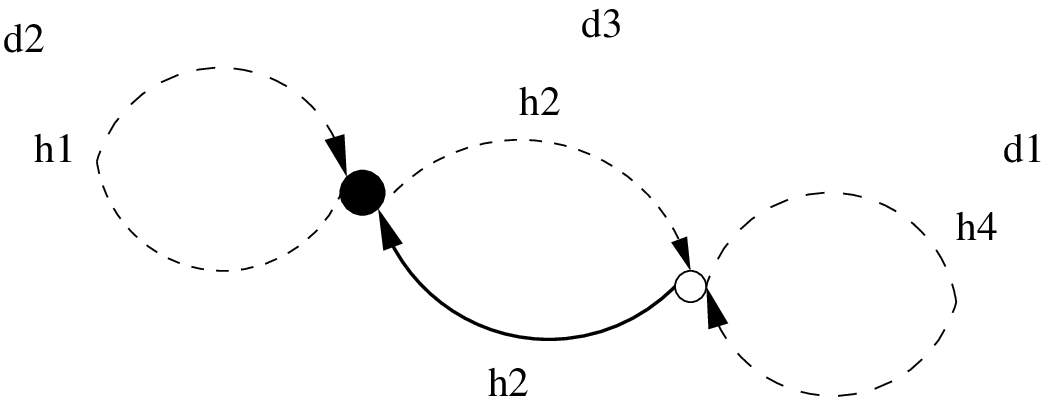}\label{fig_subgrd} } \hspace{10mm}\subfigure[Consolidated graph $\Gamma^{\prime}_D$]{
\includegraphics[scale=.5]{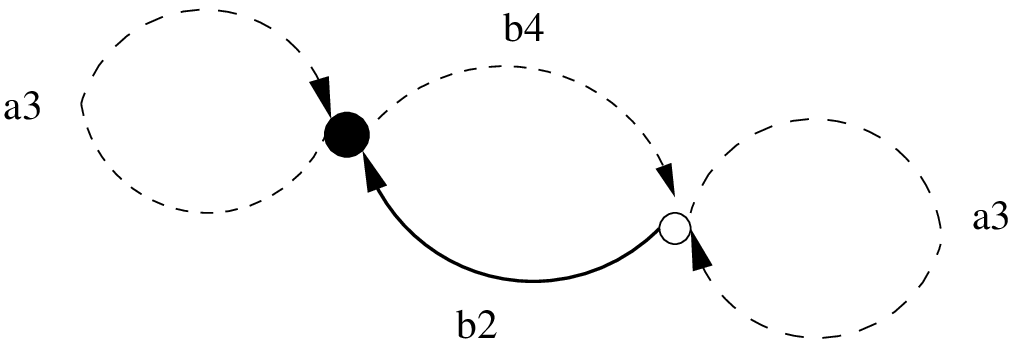}\label{fig_consd} } \hspace{10mm}
\end{center}
\caption{Subgroup graph $\Gamma_D$ and consolidated graph
$\Gamma^{\prime}_D$ for $D = \langle h_2^{-1} h_4 h_2, h_1 ,h_2^2
\rangle$}\label{fig_d}
\end{figure}

Therefore, $m_{742}, h_1$ and $a_{22}$ are central words and $1,
a_{74}, b_{42}^{-1}, b_{22}$ forms the inner part of the relative
Schreier transversal for $C$ in $F$ with respect to $C_a$ and
$\Upsilon$. This relative transversal in terms of $X = \{ a, b \}$
is shown on figure \ref{fig_consd} (and we refer reader to
\cite{3} for more details on consolidated graphs).
\end{example}

In the sequel, we omit a maximal subtree $\Upsilon$ from the
notation assuming it fixed for every relative Schreier
transversal.

In Lemma \ref{le:e(fg)} we showed that every representative $g \in
CfC$ has a unique presentation in the form $g = c_1 f c_2 $, $c_1,
c_2 \in C$ when $C_f = 1$. However, this is not the case for
essential cosets and such a presentation is not unique in general.

Nevertheless, the representatives of essential cosets can be
uniquely written in a similar way:

\begin{lemma}\label{le:ess!} Let $CfC$ be an essential coset and let $T$ be a
Schreier transversal for $C$ in $F$ with respect to $C_f$. Then
every $g \in CfC$ has a unique presentation
$$g = c f t, \,\,\,c_1 \in C, t \in T.$$
\end{lemma}

\begin{proof} Let $g = c_1 f d_1$ be a representative of a some essential double coset $CfC$, $c_1, d_1 \in C$. By definition of a relative Schreier
transversal there is an element $t_1 \in T$ such that $g = c_1 f
t_1$. Now suppose that $g= c_2 f t_2$ is another presentation for
$g$, $c_2 \in C, t_2 \in T$. Then $c_1ft_1 = c_2 f t_2$ implies
$c_1^{-1}c_2 = f t_1 t_2^{-1} f^{-1}$ and since $t_1 t_2^{-1} \in
C$ we have $c_1^{-1}c_2 \in C_f$ and $t_1 \in C_f t_2$. But $t_1,
t_2$ are both representatives of $C_f$ in $C$ and since they are
in the same coset, $t_1 = t_2$ and hence $c_1 = c_2$.
\end{proof}

\subsection{Cancellations in complex $CfC$}\label{subs_weakred}

In this subsection we estimate the size of cancellations in a
double coset depending on a subgroup $C_f$.

Let $Y = \{ h_1, \ldots, h_r  \}$ be some fixed set of Nielsen
generators for $C$. Recall that $M=\left[\frac{1}{2}max\{l_X(h_1),
\ldots, l_X(h_r) \}\right]+1$ and set $p$ equal to the
number of elements in the ball of radius $2M$ in $F(X)$. We shall
use this notation in what follows. The following technical lemma
turns out to be crucial in the proof of Theorem
\ref{th:bounded}.

\begin{lemma}\label{le:key} Let $f_0, f_1 = h_{i_1} f_0 h_{j_1},
\ldots, f_k = h_{i_k} f_{k-1} h_{j_k}$ be a sequence of elements
in $F(X)$ such that $h_{i_1}, h_{j_1}, \ldots, h_{i_k}, h_{j_k}$
are elements of $Y \cup Y^{-1}$; $h_{i_k} \dots h_{i_1}$ and
$h_{j_1} \dots h_{j_k}$  are $Y-$reduced words, and $l_X(f_l) \leq
2M$ for all $l = 1, \ldots, k$. Then for all $k \geq p$ there
exist $l,n \in \{ 1, \ldots, k \}$ and nontrivial $c_l, c_n, d_l,
d_n \in C$ such that $ l \neq n$, and $c_l\neq c_n, d_l \neq d_n$,
and
$$c_lf_ld_l= c_nf_ld_n = c_nf_nd_n= c_lf_nd_l.$$
\end{lemma}

\begin{proof} Fix a number $k$; notice that $M$ is at least $1$ so $p \geq 4
m^2 +1$ for a free group $F$ of rank $m$ and therefore $k \geq p
\geq 17.$
 In particular, $k > 0$. The number of all elements of length not greater then $2M$ is equal to $p$ by definition, and so by Dirichlet's drawer
principle there are at least two equal elements among $k+1>p$
elements of such length, say $f_l$ and $f_n$ with $l < n$. The
statement of the lemma is obvious for $f_l = f_n =1$, so suppose
$f_l , f_n$ are non-trivial. Since

$$ h_{i_k} \dots h_{i_1} f_0 h_{j_1} \dots h_{j_k} = h_{i_k} \dots h_{i_{l+1}} f_{l} h_{j_{l+1}} \dots h_{j_k}
= f_k, $$ one can take $c_l =h_{i_k} \dots h_{i_{l}}$, $c_n =
h_{i_k} \dots h_{i_{n}}$ and $d_l =h_{j_l} \dots h_{i_{k}}$, $d_n
=h_{j_n} \dots h_{i_{k}}$. Clearly, $c_l,c_n, d_l,d_n$ are
non-trivial (as they are $Y-$reduced products of non-trivial elements ) and
$c_l \neq c_n$,  $d_l \neq d_n$. Indeed, let, for instance, $c_l
=c_n$; then the product $c_l^{-1} c_n$ is equal to $h_{i{l+1}}
\dots h_{i_k}$. But $h_{i{l+1}} \dots h_{i_k}$ can not be trivial
as a $Y-$reduced word in Nielsen generators. Therefore, $c_lf_ld_l=
c_nf_ld_n$ and other equalities are straightforward. \end{proof}

Let $A_1, A_2, \ldots, A_n$ be a finite number of subsets in
$F(X)$ and consider the function $cn: A_1 \cdot A_2 \cdot \ldots
\cdot A_n \mapsto \mathbb{N}$ which computes the total amount of
$X-$cancellations in the product $a_1 \cdot \ldots \cdot a_n$:

$$cn(a_1, a_2, \ldots, a_n)
=\fracd{1}{2}\left(\mathop{\sum}\limits^{n}_{i=1} l_X(a_i) -
l_X(a_1 a_2 \cdot \ldots \cdot a_n)\right).$$

Observe that for $n=2$ the function $cn(a_1, a_2)$ coincides with
the notion of Lyndon-Chiswell-Gromov product $(a_1, a_2^{-1})$ in
a group $F$ with respect to canonical length function $l_X$ (see
\cite{13,5} for details on length functions, and in \cite{9}
this definition was adapted to hyperbolic metric spaces).
 We say that the product $A_1
\cdot A_2 \cdot \ldots \cdot A_n$ is \emph{$k-$reduced} if there
is a constant $k \ge 0$ such that for every $a_i \in A_i, i=1,
\ldots , n$ the following holds $$cn(a_1, a_2, \ldots, a_n) \leq k.$$
For instance, two singletons $A_1 = \{ u\}, A_2 = \{v \}$ such
that $uv = u\circ v$ form a $0-$reduced set, but we shall omit the
prefix $0-$ below (saving ``$k-$'' only for $k >0$).

\begin{theorem}\label{th:bounded} Let $F$ be a free group with a
finite basis $X$ and $C$ be a finitely generated subgroup of $F$
with a Nielsen set of generators $Y$. Suppose $f \in F
\smallsetminus C$ and $f$ is of minimal length in a double coset
$CfC$, and let $k = 2pM$.

\begin{itemize}
\item If $C$ is $f-$malnormal, then
$CfC$ is $k-$reduced.

\item If $CfC$ is essential, then $CfT$ is $k-$reduced, where
$T$ is the relative Schreier transversal of $C$ in $F$ relative to
$C_f$.
\end{itemize}
 \end{theorem}

\begin{proof} Let $g \in CfC$. If $f$ does not cancel completely
in $g = c f d$, then by Lemma \ref{le:34} the length of
cancellations $cn(c,f)$ is bounded above by $M$ (and also
$cn(f,d)\le M$). Hence $cn(c,f,d) = cn(c,f)+cn(f,d) \leq 2M$ and
both (1.) and (2.) follows. If, on the other hand, $f$ cancels
completely then there are the two cases $C_f = 1$ and $C_f \neq 1$ to be
considered.

{\bf 1.} Suppose $C_f= 1$ and $g \in CfC$ such that the total
amount of cancellations in $cfd = g$ is greater than $k$ for $c =
h_{i_1} \ldots h_{i_k}, d = h_{j_1} \ldots h_{j_n} \in C$.

Let $f = f' \circ f''$, where $f'$ cancels completely in $cf$ and
$f''$ cancels in $fd$. Below we list all possible forms of pieces
$f'$ and $f''$. By Lemma \ref{le:34} $f''$  possesses the following
properties:

\begin{itemize}
\item [1.1.] either $(f'')^{-1} = s$, where $s \in S_{\rm int}$
and $s$ is an initial segment of $s_{1j_1}$. Then $f'' d = \g h_{j_2}
\cdots h_{j_n}$ and $\g^{-1}$ is an initial segment of $h_{j_1}^{-1}$ of
length $l_X(\g) < l_X(h_{j_1}) \le 2M$, and $l_X(f'') \le M$;\\

\item [1.2.]  or $(f'')^{-1} = s_{1j_1} \mu_{j_1}$. Then $f'' d =
s_{2j_1}^{-1} h_{j_2} \cdots h_{j_n}$ and $l_X(s_{2j_1}^{-1}) \le
M$, $l_X(f'') \le M$;\\

\item [1.3.]  or $(f'')^{-1} = a_{j_1 j_2} \circ m_{j_1 j_2 j_3}
\circ m_{j_2 j_3 j_4} \circ \cdots \circ \circ m_{j_{s-2} j_{s-1}
j_s} \circ \a \circ \mu_{j_s}$, in which case  $f'' d = \b
h_{j_{s+1}} \cdots h_{j_n}$; recall that $\a,\b$ are elements of
$F(X)$ such that $m_{j_{s-1} j_{s} j_{s+1}}=\a \circ \mu_{j_s}
\circ \b$. Here the length $l_X(f'') \le M$ (see Lemma
\ref{le:34}, (e)) and $l_X(\b) \le M$;\\

\item [1.4.]  or  $(f'')^{-1} = a_{j_1 j_2} \circ m_{j_1 j_2 j_3}
\circ m_{j_2 j_3 j_4} \circ \cdots \circ \circ m_{j_{s-2} j_{s-1}
j_s} \circ \a_1 $, where $f'' d = \a_2 \mu_{j_s} \b h_{j_{s+1}}
\cdots h_{j_n}$; and $\a_1, \a_2, \b$ are such that $m_{j_{s-1}
j_{s} j_{s+1}}=\a_1 \a_2 \circ \mu_{j_s} \circ \b$. Here the
length $l_X(f'') \le M$ (see (f) of Lemma \ref{le:34}) and
$l_X(\a_2 \mu_{j_s} \b) \le l_X(h_{j_s}) \le 2 M$.
\end{itemize}
Similarly, $f'$ and $cf'$ have one of the forms:

\begin{itemize}
\item [2.1.] either $(f')^{-1} = s$, where $s \in S_{\rm int}$ and
$s$ is a terminal segment of $s_{2i_k}$. Then $c f' = h_{i_1} \cdots
h_{i_{k-1}}  \g $ and $\g^{-1}$ is a beginning of $h_{i_k}$ of
length $l_X(\g) \le 2M$, and $l_X(f') \le M$;\\

\item [2.2]  or $(f')^{-1} =  \mu_{i_k} s_{2i_k}^{-1}$. Then $c f' =  h_{i_1} \cdots h_{i_{k-1}} s_{1i_k} $
 and $l_X( s_{1i_k}) \le M$, $l_X(f') \le M$;\\

\item [2.3.] or $(f')^{-1} = \mu_{i_l}  \circ \b \circ m_{i_l
i_{l+1} i_{l+2}} \circ m_{i_{l+1} i_{l+2} i_{l+3}} \circ \cdots
\circ m_{i_{k-2} i_{k-1} i_k} \circ b_{i_{k-1} i_k} $, in which
case $c f' =  h_{i_1} \cdots h_{i_{l-1}} \a$;
 here $\a,\b$ are such that $m_{i_{l-1} i_{l}
i_{l+1}}=\a \circ \mu_{i_l} \circ \b$. Here the length $l_X(f')
\le M$ and $l_X(\a) \le M$;\\

\item [2.4.] or $(f')^{-1} = \b_2 m_{i_l i_{l+1} i_{l+2}} \circ
m_{i_{l+1} i_{l+2} i_{l+3}} \circ \cdots \circ m_{i_{k-2} i_{k-1}
i_k} \circ b_{i_{k-1} i_k}  $, where $c f' =  h_{i_1} \cdots
h_{i_{l-1}}   \a  \circ  \mu_{i_l}  \circ \b_1$; and $\b_1, \b_2,
\a$ are such that $m_{i_{l-1} i_{l} i_{l+1}}=\a \circ \mu_{i_l}
\circ \b_1 \b_2$. Here the length $l_X(f') \le M$ and $l_X(\a
\mu_{i_l} \b_1) \le 2M$.
\end{itemize}

Summarizing conditions (1.1) -- (2.4) and renumbering $h_{i_1},
\ldots, h_{j_n}$ for notational simplicity if necessary, we obtain

\begin{equation}\label{q1q2}
g= h_{i_1} \cdots h_{i_k} f' f''  h_{j_1} \cdots h_{j_n} = h_{i_1}
\cdots h_{i_{l-1}} q_1 q_2 h_{j_{s+1}} \cdots h_{j_n},
\end{equation}
where
$$
q_1 = \left\{
\begin{aligned}
\g_1, {\textrm{ see (2.1), and $l_X(\g_1) \le 2M$,}} \\
s_{1i}, {\textrm{ see (2.2-2.3), and $l_X(s_{1i}) \le M$,}} \\
s_{1i} \mu_i \b_1, {\textrm{ see (2.4), and $l_X(s_{1i} \mu_i \b_1) \le 2 M$}}, \\
\end{aligned}
\right.
$$
and

$$
q_2 = \left\{
\begin{aligned}
\g_2, {\textrm{ see (1.1), and $l_X(\g_2) \le 2M$,}} \\
s_{2j}^{-1}, {\textrm{ see (1.2-1.3), and $l_X(s_{2j}) \le M$,}} \\
\a_2 \mu_i s_{2j}^{-1} , {\textrm{ see (1.4), and $l_X(\a_2 \mu_i s_{2j}^{-1}) \le 2 M$}}. \\
\end{aligned}
\right.
$$

If $q_1 q_2 = q_1 \circ q_2$, then again $cn(c,f,d) = cn(c,f) +
cn(f,d)$ and the total cancellations in $cfd$ are bounded above by
$2M$, in contradiction with assumption $cn(c,f,d) > k$. Therefore,
$q_1 q_2$ is not reduced. Notice that $f \notin C$ implies $g
\notin C$ and therefore $q_1 q_2 \notin C$.  Suppose $q_1 q_2$
cancels to $q_1' \circ q_2'$, where $q_1'$ is an initial part of
$q_1$ and $q_2'$ is a terminal part of $q_2$. Without loss of
generality one can assume that all $h_{i_1}, \ldots, h_{j_n}$
cancel in
$$c f d = h_{i_1} \cdots h_{i_{l-1}} q_1' \cdot q_2' h_{j_{s+1}} \cdots h_{j_n}.$$

By definition of the function $cn$ we have

\begin{equation}\label{cn(cfd)}
cn(c,f,d) = \fracd{1}{2}\left( l_X(c) +l_X(f)+l_X(d) -
l_X(g)\right),
\end{equation}

where $$l_X(c) = \mathop{\sum}\limits_{t=1}^{k}l_X(h_{i_t}) - 2
\mathop{\sum}\limits_{t=1}^{k-1}cn(h_{i_t}, h_{i_{t+1}}),
$$
$$l_X(d) = \mathop{\sum}\limits_{z=1}^{n}l_X(h_{j_z}) - 2
\mathop{\sum}\limits_{z=1}^{n-1}cn(h_{j_z}, h_{j_{z+1}}),
$$
and $l_X(f) = l_X(f')+ l_X(f'')$.

Rearranging summands in (\ref{cn(cfd)}) and (\ref{q1q2}) we get
$$ cn(c,f,d) = cn(h_{i_{l}} \cdot \ldots \cdot h_{i_k}, f')+  cn( f'', h_{j_{1}} \cdot \ldots \cdot
h_{j_{s}})-$$
$$- \mathop{\sum}\limits_{t=1}^{l-1}cn(h_{i_t},
h_{i_{t+1}}) - \mathop{\sum}\limits_{z=s}^{n-1}cn(h_{j_z},
h_{j_{z+1}})+ $$

$$+ \fracd{1}{2} \left( l_X(q_2)+ l_X(q_1)+
\mathop{\sum}\limits_{t=1}^{l-1}l_X(h_{i_t})+
\mathop{\sum}\limits_{z=s+1}^{n}l_X(h_{j_z}) - l_X(g)\right).$$

Further, formulae (1.1) - (2.4) imply $cn(h_{i_{l}} \cdot \ldots
\cdot h_{i_k}, f') \leq M$, $cn( f'', h_{j_{1}} \cdot \ldots \cdot
h_{j_{s}}) \leq M$, and $l_X(q_1) \leq 2M, l_X(q_2) \leq 2M$. By
definition of a Nielsen set of generators $cn(h_i, h_j) \geq 0$
and $l_X(h_i) \leq 2 M$ for all $h_i, h_j \in Y \cup Y^{-1}$.
Therefore,

$$k < cn(c,f,d) \leq (l+n-s+3) M - \fracd{1}{2} l_X(g).$$

Moreover, the assumption that all $h_{i_1}, \ldots, h_{j_n}$
cancel in (\ref{q1q2}) implies $l-1 = n-s$ and

\begin{equation}\label{g=ab}
g = a_{i_1i_2}' \circ b_{j_1 j_2}',
\end{equation}

where $a_{i_1i_2}'$ is an initial segment of $a_{i_1i_2} \in
\Sigma$, $b_{j_1 j_2}'$ is a terminal segment of $b_{j_{n-1} j_n}
\in \Sigma$ and $g$ is not $C-$admissible. Hence, $(i_1, i_2) \neq
(j_{n-1}, j_n)$ and since $l_X(g)
> 0$, we have
\begin{equation}\label{lp}
0 < l_X(g) \leq 4M (l+1) - 4 p M \,\,\,\,\Rightarrow \,\,\,\, l >
p-1.
\end{equation}

Consider the sequence

\begin{align*}
f_0 &= f,\\
f_1 &= h_{i_{l}} \cdot \ldots \cdot h_{i_k} f_0 h_{j_{1}} \cdot
\ldots \cdot
h_{j_{s}},\\
f_2 &= h_{i_{l-1}} f_1 h_{j_{s+1}},\\
&\ldots,\\
f_{l-1} &= h_{i_2} f_{l-2} h_{j_{n-s}},\\
f_{l} &= g.
\end{align*}
Here $l_X(f_0) = l_X(f') + l_X(f'')$ and by (1.1) - (2.4) we have
$l_X(f_0) \le 2M$ and from formula (\ref{g=ab}) one can easily
deduce $l_X(f_{l}) \le 2M$.

Further, since $f_1 = q_1' \circ q_2'$, and $c, d$ are $Y-$reduced
words in the Nielsen set of generators $Y$ (see properties (i) -
(iii)), the equality (\ref{g=ab}) is possible only if all lengths
$l_X(f_1), \ldots, l_X(f_{l-1})$ are bounded above by $2M$.

Observe that $l > p -1$ due to (\ref{lp}), and $f_0, \ldots,
f_{l}$ represent the same double coset. Hence by Lemma
\ref{le:e(gg)}, (3.) we have $C_f = C_{f_0} \simeq \ldots \simeq
C_{f_{l}} \simeq
 C_{f_{l+1}}= C_g$.

Therefore, $f_0, \ldots, f_{l}$ satisfy the assumptions of Lemma
\ref{le:key} and hence there exist elements, say $f_i, f_j, c_i,
d_i, c_j$ and $d_j$
 such that $c_i \neq c_j, d_i \neq d_j$, but $c_i f_i d_i = c_j f_i
 d_j$, a contradiction to the uniqueness of the representation of $f_i$ in $CfC$ (see Lemma \ref{le:e(fg)}).

{\bf 2.} Let $C_f \neq 1$. Then by Lemma \ref{le:ess!} there is a
unique presentation $g = c f d$ for $c \in C$ and $d \in T$ for
the relative Schreier transversal $T=T_{C_f}$ (recall that we
fixed a maximal subtree $\Upsilon$ in $\Gamma_{C_f}$; hence $T$ is
unique). Then considering elements $d \in T$ and arguing as above,
we obtain the equality $g = c_i f_i d_i = c_j f_i d_j$ for different
indexes $i,j$ again, a contradiction with Lemma
\ref{le:ess!}.\end{proof}

\begin{remark} Notice that for essential double cosets the complex $CfC$
might not be $k-$reduced for any $k$. Indeed, let $w$ be a
primitive element of an arbitrary finite rank free group $F$, and
let $C$ be the subgroup of $F$ generated by $w^l$ for some $l>1$.
Then $f=w$ is pn and does not belong to $C$. However, the total
amount of cancellations in the product $w^{kl} w w^{-kl} \in CfC$
can be arbitrary large.
\end{remark}

\subsection{Formal language and automatic properties}\label{section_measurable}

In this subsection we investigate connections between bounded
cancellation and regularity of subsets of $F(X)$. We establish
the regularity of all reduced words representing elements of a
double coset $CfC$ in a free group $F(X)$ using this relation.

\begin{remark}\label{rem:weakred} Suppose $A_1, A_2$ are two regular subsets in $F$. Then the set
$\overline{A_1A_2}$ of all reduced words in $A_1A_2$ is regular in $F(X)$.
\end{remark}
The set $A_1A_2$ is regular by Benois' theorem (see \cite{1.5}), but we also need an automaton accepting $\overline{A_1A_2}$.
We shall show below how to construct this automaton for the case where the product $A_1A_2$ is $k-$reduced. Let $\A_i=(S(\A_i),X,\d_i,S_0(\A_i),F_0(\A_i))$
be a deterministic automaton accepting $A_i$; in particular, it means
$S_0(\A_i) = \{s_0(\A_i)\}$, $i = 1, 2$. One can easily form a
(non-deterministic) automaton $\A$, which accepts the
concatenation $A_1A_2$. Namely, take the same alphabet $X$, $S(\A)
= S(\A_1) \cup S(\A_2)$, the new start state to be $s_0(\A_1)$, the
final states to be $F_0(\A) = F_0(\A_1) \cup F_0(\A_2)$, the transition function to be $\d_1 \cup \d_2$;
and add additional arrows labelled $\varepsilon$ from all states of $F_0(\A_1)$ to $s_0(\A_2)$. Since $A_1A_2$ is
$k-$reduced, the length of cancellations between the elements $a_1
a_2$, where $a_1 \in A_1, a_2 \in A_2$ is bounded by $k$.
Therefore, the set $U$ of words $u_j \in F(X)$ such that $a_1 = b_1
\circ u_j^{-1}$, $a_2 = u_j \circ b_2$, and $a_1 a_2 = b_1 \circ
b_2$, where $a_i \in A_i$, $b_i \in F(X)$, is finite. Since $\A_1$
is a finite state automaton, for every $u_j \in U$ there is a
finite set $P_{u_j} = \{ p_{u_j} \in S(\A_1): {\textrm{ $u_j^{-1}$
is readable from }} p_{u_j} {\textrm{ to $f_0$ for some $f_0 \in
F_0(\A_1)$}}\}$; by the same argument the set $Q_{u_j} = \{
q_{u_j} \in S(\A_2): {\textrm{ $u_j$ is readable from }} s_0(\A_2)
{\textrm{ to $q_{u_j}$ }} \}$ is finite. Add
$\varepsilon-$transitions from the states of $P_{u_j}$ to the
states of $Q_{u_j}$ for all $u_j$.
 Clearly, $\overline{A_1A_2} \subseteq L(\A)$ by construction, and moreover, $\overline{A_1A_2}
= L(\A) \cap F(X)$. Thus $\overline{A_1A_2}$ is an intersection of regular sets, so is itself regular.

The {\em double-based cone } with bases $w_1, w_2$ is the set of all
reduced words $w_1 \circ f \circ w_2$ in $F(X)$, starting with $w_1$ and ending with
$w_2$.
\begin{corollary}\label{cor:cones_reg} Every double-based cone $C(w_1,w_2)$ in a finite rank free group is a regular set.
\end{corollary}
\begin{proof} One can prove this statement using Benois' theorem, but we are interested in it's direct proof and corresponding automaton. Let
$w_1 = x_{i_1} \cdots x_{i_s}$, $w_2 = x_{j_1} \cdots x_{j_t}$.
Then by definition an element $f \in C(w_1,w_2)$ is an $X-$reduced
word of the form $f = x_{i_1} \cdots x_{i_s} \circ y \cdots z
\circ x_{j_1} \cdots x_{j_t}$, with $y \neq x_{i_s}^{-1}$, and $z
\neq x_{j_1}^{-1}$. Thus, it is sufficient to prove the regularity
of a cone with bases $x_{i_s}$ and $x_{j_1}$. Indeed, if
$C(x_{i_s}, x_{j_1})$ is regular, then so is $C(w_1, w_2) =
x_{i_1} \cdots x_{i_{s-1}} \circ C(x_{i_s}, x_{j_1}) \circ x_{j_2}
\cdots x_{j_t}$ as it is a concatenation of regular sets in
$F(X)$. For notational simplicity we construct a deterministic
automaton $\A$ recognizing the cone $C(x_{i_s}, x_{j_1})$ in $F(X) =
F(a,b)$, for bases equal to $a$ and $b$ respectively. An
automaton $\A$ recognizing $C(a,b)$ is shown in Figure
\ref{fig_Cab} (the tailed arrow corresponds to the initial state plus $a-$transition, and the final state is labelled by double circle).

\begin{figure}[h!]
\begin{center}

\psfrag{a}{$a$}

\psfrag{b}{$b$}

\psfrag{a1}{$a^{-1}$}

\psfrag{b1}{$b^{-1}$}

\includegraphics[width=6cm]{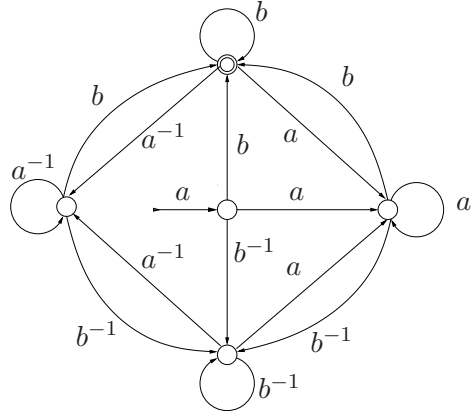}\\
\caption{The automaton $\A$ recognizing the cone $C(a,b)$.}\label{fig_Cab}
\end{center}
\end{figure}

If now the rank of the free group $F$ is greater than $2$, $\A$ can be easily modified to recognize $C(a,b)$ in $F=F(X)$. Namely, it
is necessary to add a new state $q_x$ corresponding to each element of $X \cup X^{-1}$ (except $a^{\pm 1}, b^{\pm 1}$). From every new state $q_{x}$ we add a $y-$transition to every other state $q_y$, including $q_x$, but excluding $q_{x^{-1}}$. The arrows from basis $a$ and to basis $b$ can be added in similar way.
 \end{proof}

\begin{theorem}\label{th:cfcreg} Let $C$ be a finite rank subgroup of a free
group $F(X)$. Then the set of all $X-$reduced words representing
elements of $CfC$ is regular in $F(X)$. Moreover, an automaton which accepts this set can be constructed effectively by $C$ and $f$ if $C$ is $f-$malnormal, or by $C,f$ and a relative Schreier transversal $T$ of $C$ with respect to $C_f$, otherwise.
\end{theorem}

\begin{proof} If $C_f = 1$, then every $X-$reduced word in $CfC$ has a unique representation in terms of this complex, by Lemma \ref{le:e(fg)}.
If $C_f \neq 1$, then every word in $CfC$ can be represented by an element of the complex $CfT$, for a relative Schreier transversal
$T$ of $C$ with respect to $C_f$, by Lemma \ref{le:ess!}. In both cases the sets $CfC$ and $CfT$ are regular: it follows either from Benois' theorem, or from the following direct construction. The sets $CfC$ and $CfT$ are $k-$reduced in $F(X)$, by Theorem \ref{th:bounded}. Further, since $C$, $fC$ and $fT$ are regular in
$F(X)$ and corresponding automata can be constructed effectively (see, for example, \cite{7}), the result follows from Remark \ref{rem:weakred} and it's proof.
\end{proof}

\section*{Acknowledgments}
The first author is grateful to A.~Duncan, A.~A.~Klyachko and V.~A.~Romankov for useful remarks.
The second author was partially supported by RFFI grant 11-01-00081.


\begin{thebibliography}{0}

\bibitem{1} G.~Baumslag, A.~G.~Miasnikov, V.~N.~Remeslennikov, Malnormality is decidable in free groups, {\it Intern. J. Algebra
Comput.} {\bf 9}(6) (1999) 687-692.

\bibitem{1.5} M.~Benois, Parties rationnelles du groupe libre, {\it C.R. Acad. Sci. Paris, Ser. A}, {\bf 269} (1969)

\bibitem{2} M.~Bestvina, M.~Feighn, A combination theorem for negatively
curved groups, {\it Intern. J. Diff. Geom.} {\bf 35}(1) (1992) 85-101.

\bibitem{3}  A.~V.~Borovik, A.~G.~Myasnikov and V.~N.~Remeslennikov,
The Conjugacy Problem in Amalgamated Products I: Regular
Elements and Black Holes, {\it Intern. J. Algebra
Comput.} {\bf 17}(7) (2007) 1301 - 1335.

\bibitem{4} M.~Bridson and D.~Wise, Malnormality is undecidable in hyperbolic
groups, {\it Israel J. Math.} {\bf 124} (2001) 313-316.

\bibitem{5} I.~Chiswell, Abstract length functions in groups, {\it Math. Proc. Cambridge Phil. Soc.} {\bf 80} (1976) 451--463.

\bibitem{6} D.~Epstein, J.~Cannon, D.~Holt, S.~Levy, M.~Paterson and
W.~Thurston, {\it Word Processing in Groups} (Jones and Bartlett,
Boston, 1992).

\bibitem{7} E.~Frenkel, A.~G.~Myasnikov and V.~N.~Remeslennikov,
Regular sets and counting in free groups, in {\it Combinatorial and
Geometric Group Theory}, Series ``Trends in Mathematics'', (Birkhauser Verlag Basel/Switzerland, 2010),
pp.~93--118.

\bibitem{8}  E.~Frenkel, A.~G.~Myasnikov and V.~N.~Remeslennikov,
Amalgamated products of groups: measures of random normal
forms, {\it Fund. Appl. Math.} {\bf 16}(8) (2010) 189-221.

\bibitem{9} M.~Gromov, {\it Hyperbolic groups}, in Essays in Group Theory, Mathematical Sciences
Research Institute Publications, Vol. 8 (Springer-Verlag, Berlin, 1987), pp.~75--263.

\bibitem{10} T.~Jitsukawa, Malnormal subgroups of free
groups, in {\it Computational and statistical group theory}, Series on Contemporary Mathematics, Vol. 298, Amer. Math. Soc.,
Providence, RI, (2002)  pp.~83--95.

\bibitem{11} I.~Kapovich and A.~G.~Myasnikov, Stallings foldings and
subgroups of free groups, {\it J. Algebra} {\bf 248} (2002) 608-668.

\bibitem{12} O.~Kharlampovich and A.~Myasnikov, Hyperbolic groups and free
constructions, {\it Trans. Amer. Math. Soc.} {\bf 350}(2) (1998) 571-613.

\bibitem{13} R.~Lyndon, {\it Length functions in groups}, Math. Scand., {\bf 12} (1963) 209 -- 234.

\bibitem{14} R.~C.~Lyndon and P.~Schupp, {\it Combinatorial group
theory} (Ergebnisse der Mathematik und ihrer Grenzgebiete Vol. 89,
Springer-Verlag, Berlin, Heidelberg, New York, 1977).

\bibitem{15} K.~V.~Mikhajlovskii and A.~Yu.~Ol'shanskii, Some
constructions relating to hyperbolic groups, in {\it Geometry and Cohomology in Group Theory,} London Math. Soc.
Lecture Notes Series, Vol. 252 (Cambridge University Press, Cambridge, 1998), pp.~263--290.

\bibitem{16} N.~W.~M.~Touikan, A Fast Algorithm for Stallings' Folding Process, {\it  Intern. J. Algebra
Comput.} {\bf 16}(6) (2006) 1031-1046.




\end{thebibliography}
\end{document}